\documentclass[twoside,11pt,a4paper]{amsart}
\usepackage{amssymb,supertabular,a4}
\usepackage{amsmath,amsthm,amscd,amssymb,graphicx}
\input xy
\xyoption{all}

\def\mod{\hbox{\rm mod\,}}

\def\udim{\hbox{\rm \underline{dim}\,}}
\def\ra{{\rightarrow}}

\title[]{The GR-segments for tame quivers}

\author{Bo Chen}
\date{}
\address {Universit\"at zu K\"oln\\
          Mathematisches Institut\\
                   Weyertal 86-90\\
           D-50931 K\"oln\\ Germany}

\email {mcebbchen@googlemail.com}

\thanks{The author is supported by DFG-Schwerpunktprogramm 1388
`Darstellungstheorie'.}

\newtheorem{theo}{Theorem}[section]

\newtheorem{lemm}[theo]{Lemma}
\newtheorem{coro}[theo]{Corollary}
\newtheorem{prop}[theo]{Proposition}

\newtheorem*{theo*}{Theorem}

\begin{document}

\begin{abstract}
A GR-segment for an artin algebra is a
sequence of Gabriel-Roiter measures, which is closed under
direct predecessors and successors.  The number of the GR-segments indexed by 
natural numbers $\mathbb{N}$ and  integers $\mathbb{Z}$ probably relates to the representation types
of artin algebras.  Let $k$ be an algebraically closed field and $Q$ be a tame quiver (of type
$\widetilde{\mathbb{A}}_n$, $\widetilde{\mathbb{D}}_n$, $\widetilde{\mathbb{E}}_6$,
$\widetilde{\mathbb{E}}_7$, or $\widetilde{\mathbb{E}}_8$). Let $b$ be the number of the isomorphism classes of the exceptional quasi-simple modules over the path algebra $\Lambda=kQ$.
We show that the number of the $\mathbb{N}$- and
$\mathbb{Z}$-indexed GR-segments in the central part for $Q$ is bounded by $b+1$. 
Therefore, there are at most $b+3$ GR segments.
\end{abstract}

\maketitle



\section{Preliminaries and main theorem}

We fist  recall what Gabriel-Roiter measures are
\cite{R1,R2}. Let $\mathbb{N}$=$\{1,2,\ldots\}$ be the set of
natural numbers and $\mathcal{P}(\mathbb{N})$ be the set of all
subsets of $\mathbb{N}$.  A total order on
$\mathcal{P}(\mathbb{N})$ can be defined as follows: if $I$,$J$
are two different subsets of $\mathbb{N}$, write $I<J$ if the
smallest element in $(I\backslash J)\cup (J\backslash I)$ belongs to
J. Also we write $I\ll J$ provided $I\subset J$ and for all elements
$a\in I$, $b\in J\backslash I$, we have $a<b$. We say that $J$ {\it
starts with} $I$ if $I=J$ or $I\ll J$.

Let $\Lambda$ be a connected artin algebra and  $\mod\Lambda$ be the category
of finite generated left $\Lambda$-modules. We denote by $|M|$ the length
of a $\Lambda$-module $M$.  For each
$M\in\mod\Lambda$, let $\mu(M)$ be the maximum of the sets
$\{|M_1|,|M_2|,\ldots, |M_t|\}$, where $M_1\subset M_2\subset \ldots
\subset M_t$ is a chain of indecomposable submodules of $M$. We call
$\mu(M)$ the {\it Gabriel-Roiter} (GR for short) {\it measure}  of $M$.
If $M$ is an
indecomposable $\Lambda$-module, we call an inclusion $T\subset M$
with $T$ indecomposable a {\it GR inclusion} provided
$\mu(M)=\mu(T)\cup\{|M|\}$, thus if and only if every proper
submodule of $M$ has Gabriel-Roiter measure at most $\mu(T)$. In
this case, we call $T$ a {\it GR submodule} of $M$.

An element $I\in\mathcal{P}(\mathbb{N})$ is called a GR measure for $\Lambda$ if
there is an indecomposable $\Lambda$-module $M$ with $\mu(M)=I$.
Given a GR measure $I$,
we denote by $\mathcal{A}(I)$ the set of representatives of (the isomorphism classes) of the
indecomposable
modules with GR measure $I$.
We also denote by
$|I|$ the maximal element of $I$, i.e.,
the length of $M$ with $M\in\mathcal{A}(I)$.
The following is a direct consequence of the definitions.
\begin{lemm}\label{basic}Let $I<I'<J$  be GR measures for $\Lambda$.
\begin{itemize}
  \item[(1)] If $J$ starts with $I$, then $I'$ starts with $I$.
  \item[(2)] If $J=I\cup\{|J|\}$, then $|I'|>|J|$.
\end{itemize}
\end{lemm}

In \cite{R2}, the following theorem was proved:
\begin{theo}\label{partition} Let $\Lambda$ be a representation-infinite artin algebra.
Then  there are Gabriel-Roiter
measures $I_i$ and $I^i$:
 $$I_1<I_2<I_3<\ldots\quad\quad \ldots<I^3<I^2<I^1$$
such that any other GR measure $I$ satisfies $I_i<I<I^i$ for all $i$.
\end{theo}
The GR measures $I_i$ (resp. $I^i$) are called {\it take-off} (resp. {\it landing}) measures.
Any other
GR measure is called a {\it central measure}. An indecomposable module $M$
is called a take-off (resp.
central, landing) module if $\mu(M)$ is a take-off (resp. central, landing) measure.

Let $I$ and $J$ be two GR measures for $\Lambda$.
Then $J$ is called a {\it direct successor} of $I$ if first, $I<J$
and second, there is no other
GR measure $I'$ such that $I<I'<J$.  The so-called Successor Lemma
in \cite{R3} claims that any GR measure
different from $I^1$, the maximal one, has a direct successor.
However, a GR measure, which is not the minimal one $I_1$,
may not admit a direct predecessor.

A sequence of GR measures for  $\Lambda$ is called a {\it GR segment} if it is closed under
taking direct predecessors and successors. By Theorem \ref{partition} and the Successor Lemma,
a GR-segment $\mathcal{S}$ is finite if and only it is the only GR segment,
and  if and only if $\Lambda$ is of finite representation type.

Fix a representation-infinite artin algebra.
Starting with a GR measure $\mu_0$, we may obtain a sequence of GR measures
by taking direct successors and predecessors:
$$\ldots<\mu_{-3}<\mu_{-2}<\mu_{-1}<\mu_0<\mu_1<\mu_2<\mu_3<\ldots$$
If $\mu_0$ is  not a landing measure, then $\mu_i$ exist for all $i\geq 1$ by Successor Lemma.
However, $\mu_{-j}$ may not exist for some $r\geq 1$ and any  $j\geq r$, since there are
GR measures admitting no direct predecessors.
A infinite GR segment can be naturally said to be indexed by natural numbers $\mathbb{N}$,
-$\mathbb{N}$ or by integers $\mathbb{Z}$. 

From now on, a GR segment always means an infinite one.  The following observations are straightforward:
\begin{itemize}
\item The unique $-\mathbb{N}$-indexed GR segment is the landing part.
\item The GR-segment containing a take-off measure is $\mathbb{N}$-indexed.
\item The $\mathbb{N}$-indexed GR segments one-to-one correspond
         to the GR measures admitting no direct predecessors.
\item A GR segment containing a central measure is either $\mathbb{N}$- or $\mathbb{Z}$-indexed.
\end{itemize}

The number of the $\mathbb{N}$- and $\mathbb{Z}$-indexed GR segments
was thought to relate the representation types of finite dimensional
algebras (or more general, artin algebras) \cite{Ch4,Ch5}.
It was conjectured that
a quiver is of wild type if and only if  there are infinitely many
$\mathbb{N}$- or $\mathbb{Z}$-indexed GR segments.
It was shown in \cite{Ch4} that for a tame quiver (of type
$\widetilde{\mathbb{A}}_n$, $\widetilde{\mathbb{D}}_n$, $\widetilde{\mathbb{E}}_6$,
$\widetilde{\mathbb{E}}_7$, or $\widetilde{\mathbb{E}}_8$) there are,
but only finitely many,  GR measures
admitting no direct predecessors. This precisely means that the number
of $\mathbb{N}$-indexed GR segments is finite. It was also proved in \cite{Ch5} that for
wild $n$-Kronecker quivers there are infinitely many $\mathbb{N}$-indexed GR segments.

From now on, let $k$ be an algebraically closed field and $Q$ be tame quiver.
We refer to \cite{ARS,DR,R1} for basic concepts of representation theory of (tame) quivers.
Let $X$ be a quasi-simple module. We denote by $R_X$ the rank of $X$,
i.e., the minimal natural number such that $\tau^{R_X}X\cong X$, where $\tau$ is the Auslander-Reiten 
translation.
Any indecomposable regular module $M$ is of the form $X_i$,
where $X$ is quasi-simple and $i$ is the quasi-length of $M$,
i.e., the length of the unique sequence of irreducible monomorphisms
$X=X_1\ra X_2\ra\ldots\ra X_i=M$. Let $M=X_i$ for some quasi-simple module $X$.
$M$ is called exceptional if $R_X\geq 2$.
Otherwise, $M$ is called homogeneous and denote by $H_i$.
If $X$ is quasi-simple, the dimension vector $\udim X_{R_X}=\delta$, where $\delta$
is the minimal imaginary root of $Q$. We also denote by $|\delta|$ the sum of all coordinates of $\delta$.
Thus it is the length $X_{R_X}$. Let $b$ be the number of the isomorphism classes of
the exceptional quasi-simple modules and $a$ be the number of the
isomorphism classes of the exceptional quasi-simple modules $X$
whose GR measures satisfy $\mu(X_{R_X})\geq \mu(H_1)$.
We list the value of $b$ as follows, where $p$ is the number
of the clockwise arrows and $q$ is the number of anti-clockwise
arrows of type $\widetilde{\mathbb{A}}_{p,q}$:

\begin{center}
\begin{tabular}{|c|c|c|c|c|c|}
\hline
& $\widetilde{\mathbb{A}}_n=\widetilde{\mathbb{A}}_{p,q}$ & $\widetilde{\mathbb{D}}_n$
 & $\widetilde{\mathbb{E}}_6$ & $\widetilde{\mathbb{E}}_7$ & $\widetilde{\mathbb{E}}_8$\\
\hline
b & \begin{tabular}{c|c|c}$p=q=1$ & $p=1,q>1$ or $q=1,p>1$ & $p,q>1$\\\hline 0 & $p+q-1$ & $p+q$ \end{tabular}
  & $n+2$ & $8$ & $9$ & $10$\\
\hline
\end{tabular}
\end{center}

In this paper, we will again focus on tame quivers and study the
structure of $\mathbb{N}$- and $\mathbb{Z}$-indexed
GR segments.
The following theorem will be proved:
\begin{theo*} Let $Q$ be a tame quiver. The number of the $\mathbb{Z}$-indexed GR segments is
bounded by $a$.  The number of the $\mathbb{N}$ and $\mathbb{Z}$-indexed GR segments in the central part
is bounded by $b+1$.
\end{theo*}

The direct successors of the GR measures of regular modules will be described in  Section \ref{reg}.
Section \ref{seq} is devoted to a discussion of the structure of $\mathbb{N}$- and $\mathbb{Z}$-indexed
GR segments and a proof of the main theorem.

\section{Direct successors of GR measures of regular modules}\label{reg}

In this section, we  study the direct successors
of $\mu(X_i)$, where $X$ is a quasi-simple module and $i$ large enough.
The results in the section were first shown for quivers of type $\widetilde{\mathbb{A}}_n$
in \cite{Ch4} and claimed  being true for all tame quivers. We include the proofs
for the convenience for later discussion.
Throughout this section, we fix a tame quiver $Q$.

We collect some known facts in the following proposition, which will
be quite often used in our later discussion. The proofs can be found
in \cite{Ch3}.
\begin{prop}\label{bigprop}
\begin{itemize}

    \item[(1)] If $M$ is an indecomposable preprojective module, then $M$ is a take-off module and $\mu(M)<\mu(H_1)$.
    \item[(2)] Let $H_1$ be a homogeneous quasi-simple module. Then $\mu(H_1)$ is a central measure and $\mu(H_{i+1})$ is a direct successor
             of $\mu(H_i)$ for each $i\geq 1$. Moreover, there are only finitely many indecomposable preinjective modules
              $M$ with $\mu(M)<\mu(H_1)$.
    \item[(3)]  Let $X$ be quasi-simple and $T$  be a GR submodule of
               $X_i$ for some $i\geq 1$. Then $T$ is either preprojective or $T\cong X_{i-1}$.
    \item[(4)] Let $X$ be a quasi-simple module.
               \begin{itemize}
                   \item[a)] If $\mu(X_r)< \mu(H_1)$, then
                              $\mu(X_i)<\mu(H_j)$ for all $i\geq 1$ and $j\geq 1$.
                   \item[b)] If $\mu(X_r)\geq \mu(H_1)$, then $X_{i-1}$
                              is the unique (up to isomorphism) GR
                              submodule of $X_{i}$ for every $i\geq r$. If, in addition, $r>1$, then                              $\mu(X_i)>\mu(H_j)$ for all $i>r$ and $j\geq
                              1$.
                \end{itemize}
    \item[(5)] Let $M$ be preinjective, which is not in take-off part. If $X_i$ is
               a GR submodule of $M$ for some quasi-simple module $X$.  Then $\mu(M)>
                \mu(X_j)$ for all $j\geq 1$.
\end{itemize}
\end{prop}

\begin{lemm}\label{2delta} Let $X,Y$ be quasi-simple modules with rank $r$ and $s$,
respectively. Assume that $\mu(X_r)\geq \mu(H_1)$.
\begin{itemize}
       \item[(1)]  If $\mu(X_r)>\mu(Y_s)$, then $\mu(X_i)>\mu(Y_j)$
                 for all $i\geq r$, $,j\geq 1$.
      \item[(2)]  If $\mu(X_i)=\mu(Y_j)$ for some $i\geq 2r$, then $r=s$
                 and $\mu(X_t)=\mu(Y_t)$ for every $t\geq r$.
       \item[(3)] If $\mu(X_{2r})>\mu(Y_{2s})$, then
                 $\mu(X_i)>\mu(Y_j)$ for all $i\geq 2r, j\geq 1$.
\end{itemize}
\end{lemm}
\begin{proof}(1) If $\mu(Y_s)<\mu(H_1)$, then $\mu(Y_j)<\mu(H_1)$ for
all $j\geq 1$.  Thus we may assume that $\mu(Y_s)\geq\mu(H_1)$.
Since for each $j\geq s$, $\mu(Y_j)$ starts with $\mu(Y_s)$  and
$|Y_s|=|X_r|=|\delta|$, we have $\mu(X_r)>\mu(Y_j)$.

(2) It is clear that $r=1$ if and only if $s=1$. Now we assume $r>1$.
Since $\mu(X_r)\geq \mu(H_1)$, we have $\mu(Y_s)\geq \mu(H_1)$. Thus
$j\geq 2s$ and
$$\begin{array}{rclll}
\mu(Y_{j})&=&\mu(Y_s)\cup\{|Y_{s+1}|,|Y_{s+2}|,\ldots,|Y_{2s}|,|Y_{2s+1}|,\ldots,|Y_{j}|\}&&\\
&=&\mu(X_r)\cup\{|X_{r+1}|,|X_{r+2}|,\ldots,|X_{2r}|,|X_{2r+1}|,\ldots,|X_{i}|\}&=&\mu(X_i).
\end{array}$$
Because $|X_r|=|Y_s|=|\delta|$ and $|X_{2r}|=|Y_{2s}|=2|\delta|$, we
obtain that $r=s$, $\mu(X_r)=\mu(Y_s)$ and
$\mu(X_{2r})=\mu(Y_{2s})$. Note that
$$|X_{r+l}|-|X_{r+l-1}|=|Y_{r+l}|-|Y_{r+l-1}|$$ for all $l\geq 1$.
It follows $\mu(X_t)=\mu(Y_t)$ for all $t\geq r=s$.

(3) follows similarly.
\end{proof}

\begin{coro}\label{ds1}Let $X$ be a quasi-simple module of rank $r$ such that
$\mu(X_r)\geq \mu(H_1)$. If $M$ is an indecomposable module such
that $\mu(M)=\mu(X_i)$ for some $i\geq 2r$, then $M$ is a regular
module.
\end{coro}

\begin{proof}
For the purpose of a contradiction, let $M$ be an indecomposable preinjective
module with $|M|$  minimal such that
$\mu(M)=\mu(X_i)$ for some $i\geq 2r$.   Note that $i-1\geq 2r$, since
$|M|\neq 2|\delta|$.
Let $T$ be a GR submodule of $M$. Then $\mu(T)=\mu(X_{i-1})>\mu(H_1)$.
By the minimality of $|M|$, $T$ is regular, say $T=Y_t$ for
some quasi-simple module $Y$ of rank $s$.
Then $\mu(M)>\mu(Y_j)$ for all $j\geq 1$ by Proposition
\ref{bigprop}(5).  Thus $Y\ncong X$ and $t\geq 2s$ since
$|M|=|X_i|>2|\delta|$. It follows that $\mu(Y_s)\geq \mu(H_1)$.
Notice that $\mu(Y_t)=\mu(X_{i-1})$. Therefore, $r=s$ and
$\mu(Y_{t+1})=\mu(X_i)$ by Lemma \ref{2delta} which implies $|Y_{t+1}|=|X_i|=|M|$.
On the other hand, it is easily seen that
$|Y_{t+1}|>|M|$. This is a contradiction.
\end{proof}

\begin{prop}\label{ds} Let $X$ be a quasi-simple module of rank $r>1$.
\begin{itemize}
\item[(1)] If $\mu(X_r)\geq \mu(H_1)$. Then $\mu(X_{j+1})$ is a direct successor
of $\mu(X_{j})$ for each $j\geq 2r$.
\item[(2)] If $\mu(X_r)<\mu(H_1)$ and if there is an $i\geq 1$ such that $X_i$ is a
central module. Then there is an $i_0\geq i$ such that
$\mu(X_{j+1})$ is a direct successor of $\mu(X_j)$ for each $j\geq
i_0$.
\end{itemize}
\end{prop}

\begin{proof}
(1) We first show that there does not exist an
indecomposable regular module $M$ such that $\mu(M)$ lies between
$\mu(X_j)$ and $\mu(X_{j+1})$ for any $j\geq 2r$. For the purpose of
a contradiction, we assume that there exists a $j\geq 2r$ and an
indecomposable regular module $M$ with $|M|$
minimal and $\mu(X_j)<\mu(M)<\mu(X_{j+1})$. Then $|M|>|X_{j+1}|>2|\delta|$, since $X_j$ is a GR
submodule of $X_{j+1}$.   Let $M=Y_t$ for some quasi-simple module
$Y$ of rank $s>1$. It follows that $\mu(Y_s)\geq \mu(H_1)$ and $t>
2s$. Therefore, $Y_{t-1}$ is a GR submodule of $Y_i$ and
$$\mu(Y_{t-1})\leq \mu(X_j)<\mu(M)=\mu(Y_t)<\mu(X_{j+1})$$
by minimality of $|M|$. This implies $\mu(Y_{t-1})=\mu(X_j)$, since
otherwise $|X_j|>|M|>|X_{j+1}|$, which is impossible. Observe
that $t-1\geq 2s$ and $j\geq 2r$.  Then Lemma \ref{2delta} implies
$\mu(X_i)=\mu(Y_i)$ for all $i\geq r=s$. This contradicts the
assumption $\mu(X_j)<\mu(M)=\mu(Y_t)<\mu(X_{j+1})$. Therefore, there
are no indecomposable regular modules $M$ satisfying
$\mu(X_j)<\mu(M)<\mu(X_{j+1})$ for any $j\geq 2r$.

Assume that $M$ is an indecomposable preinjective module such that
$\mu(X_j)<\mu(M)<\mu(X_{j+1})$ with $|M|$ minimal. Let $N$ be a GR
submodule of $M$. Comparing the lengths, we have
$\mu(X_j)\leq\mu(N)$. If $N=Y_h$ is regular for some quasi-simple
module $Y$ of rank $s$, then
$\mu(X_{j+1})>\mu(M)>\mu(Y_{h+1})>\mu(Y_h)\geq\mu(X_j)$. This  contradicts the
first part of the proof. If $N$ is preinjective, then
$\mu(N)=\mu(X_j)$ by the minimality of $|M|$. Thus a GR filtration
of $N$ contains a regular module $Z_{2t}$ for a quasi-simple module $Z$ of
rank $t$. It follows that $\mu(X_{2r})=\mu(Z_{2t})$.  Thus
$\mu(M)>\mu(N)>\mu(Z_{i+1})=\mu(X_{i+1})$, which is a contradiction.

(2)
Since there are only finitely many indecomposable preinjective modules
with GR measures
smaller than $\mu(H_1)$, we may choose $j_0\geq i$ such that
$\mu(X_j)<\mu(M)<\mu(X_{j+1})$ implies
that $M$ is regular  for any $j\geq j_0$.
It is sufficient to
show that there is an $i_0\geq j_0$ such that there does not exist a regular
module with GR measure $\mu$ satisfying $\mu(X_j)<\mu<\mu(X_{j+1})$
for any $j\geq i_0$.

Since $X_i$ is a central module,  $X_j$ is the unique, up to
isomorphism, GR submodule of $X_{j+1}$ for every $j\geq j_0$.
Let $Y$ be a quasi-simple module  of rank $s$ such that
$\mu(X_j)<\mu(Y_l)<\mu(X_{j+1})$ for some $j\geq j_0\geq r$ and $l\geq
1$. In this case, $Y_l$ is a GR submodule  of $Y_{l+1}$ since $Y_l$
is a central module. Comparing the lengths, we have
$\mu(Y_{l+1})<\mu(X_{j+1})$, and similarly $\mu(Y_h)<\mu(X_{j+1})$
for all $h\geq 1$. Now replace $j$ by some $j'>j$ and repeat the
above consideration. Since there are only finitely many quasi-simple
modules $Z$ such that $\mu(Z_{R_Z})\leq \mu(H_1)$, where $R_Z$ is
the rank of $Z$, we may obtain an index $i_0>j_0$ such that a GR measure
$\mu$ of an indecomposable regular module satisfies either
$\mu<\mu(X_{i_0})$ or $\mu>\mu(X_j)$ for all $j\geq 1$.
\end{proof}

\section{The structures of GR-segments}\label{seq}

In this section, we study the structure of the
$\mathbb{N}$- and $\mathbb{Z}$-indexed
GR segments for a fixed tame quiver $Q$.
The main theorem will be also proved in this section.

\subsection{Sequence of direct successors}
Let $\mu_0$ be a central measure and $\mathcal{S}$ be the sequence of GR measures
obtained by taking direct successors starting with $\mu_0$:
$$\mu_0<\mu_1<\mu_2<\mu_3<\mu_4\ldots$$

\begin{lemm}For each $\mu\in\mathcal{S}$, $\mu>I_i$ for all
take-off measures $I_i$.  In particular,
$M$ is not a preprojective module for any $M\in\mathcal{A}(\mu)$.
\end{lemm}
\begin{proof}This is straightforward since $\mu$ is not a
take-off measure and all indecomposable preprojective modules are take-off modules
(Proposition \ref{bigprop}).
\end{proof}

\begin{lemm}\label{s=1}
Let $X$ be a quasi-simple module with $\mu(X_s)=\mu_i\in\mathcal{S}$.
Assume that $N$ is an indecomposable regular module with $\mu(N)=\mu_j$
for some $j<i$ such that $\mathcal{A}(\mu_h)$ contains no regular modules for any $j<h<i$.
Then $\mu(N)=\mu(X_{s-1})$. In particular, if $s=1$,
then $\mathcal{A}(\mu_j)$ contains no regular modules for any $j<i$.
\end{lemm}
\begin{proof}
Assume that $N=Y_t$ for some quasi-simple module $Y$ and $t\geq 1$.
Let $T\subset X_s$ be a GR submodule. Then $T$ is either preprojective,
or isomorphic to $X_{s-1}$.
By the choice of $i$ and $j$ and the fact that $\mathcal{S}$
contains no take-off measure,  we have $\mu(T)\leq \mu(Y_t)<\mu(X_s)$.
If the equality does not hold (for example, $s=1$ and thus $T$ is preprojective),
then $|Y_t|>|X_s|$ by Lemma \ref{basic} since $T$ is a GR submodule of $X_s$ .
It follows from the assumption that $\mu(Y_t)<\mu(X_s)<\mu(Y_{t+1})$.
Again consider a GR submodule of $Y_{t+1}$. Similar to the above situation,
we have $|X_s|>|Y_{t+1}|$, which contradicts $|Y_t|>|X_s|$. Thus $\mu(T)=\mu(Y_t)$
and  $T\cong X_{s-1}$ since $Y_t$ is a central module.
\end{proof}

As a direct consequence of this lemma, we can show the existence of an $\mathbb{N}$-indexed
GR segment which is not the take-off part.

\begin{coro}\label{homo}
A GR segment containing $\mu(H_i)$ is indexed by $\mathbb{N}$.
\end{coro}

\begin{proof} It is known that $\mu(H_{i+1})$ is a direct successor of
$\mu(H_i)$ for all $i\geq 1$. Thus a GR segment $\mathcal{S}_{\mathbb{Z}}$ contains
$\mu(H_i)$ for some $i$ if and only if it contains all $\mu(H_i)$.
Without loss of generality, we may assume $\mu(H_1)=\mu_0\in\mathcal{S}_{\mathbb{Z}}$.
By Lemma \ref{s=1}, for each GR measure $\mu_{-j}$ obtained by taking direct predecessors
from $\mu_0$ contains only preinjective modules. Thus
there are infinitely many indecomposable preinjective modules
with GR measures smaller than $\mu(H_1)$.  This is a contradiction.
\end{proof}

{\bf Remark.} For a tame quiver of type $\widetilde{\mathbb{A}}_n$, $\mu(H_1)$ does always not admit
a direct predecessors \cite{Ch4}.

\begin{lemm}\label{startwith} Assume that $|\mu_0|\leq |\mu_i|$ for all $i\geq 0$.
Then for each $i\geq 1$, $\mu_i$ starts with $\mu_0$.
\end{lemm}

\begin{proof} We use induction on $i$.    Since $\mu_1\neq I_1$,
we may write $\mu_1=\mu'_1\cup\{|\mu_1|\}$. Thus $\mu'_1\leq \mu_0<\mu_1$.
If the equality does not hold, then $|\mu_0|>|\mu_1|$.
This contradicts the minimality of $|\mu_0|$. Thus $\mu'_1=\mu_0$ and
$\mu_1$ starts with $\mu_0$.  Now assume that $i>1$ and $\mu_r$ starts
with $\mu_0$ for all $1\leq r\leq i$. Let $\mu_{i+1}=\mu'_{i+1}\cup\{|\mu_{i+1}|\}$.
If $\mu'_{i+1}\leq \mu_0<\mu_{i+1}$, then we are done.
Otherwise, $\mu'_{i+1}=\mu_r$ for some $1\leq r\leq i$.
Hence, $\mu'_{i+1}$ and thus $\mu_{i+1}$ starts with $\mu_0$ by induction.
\end{proof}

\begin{lemm}\label{regular1}  For each $i$, there is some $j>i$ such that
$\mathcal{A}(\mu_j)$ contains regular modules.
\end{lemm}

\begin{proof} To obtain a contradiction, we may assume,
without loss of generality, that $|\mu_0|\leq|\mu_i|$ and that $\mathcal{A}(\mu_i)$
contains only preinjective modules for all $i\geq 0$. By previous lemma, $\mu_i$
starts with $\mu_0$ for all $i\geq 1$. Since $\mathcal{A}(\mu_0)$ contains only
finitely many indecomposable preinjective modules, there are infinitely many indecomposable preinjective modules
containing a preinjective module $M\in\mathcal{A}(\mu_0)$ as a submodule.  This is impossible.
\end{proof}

Namely, we may show a much stronger consequence.

\begin{lemm} \label{regular2} There is some $i$ such that $\mathcal{A}(\mu_j)$ contains only
regular modules for all $j\geq i$.
\end{lemm}

\begin{proof}
Since for any indecomposable preinjective module $N$,  $\mu(N)\neq\mu(H_j)$ for any $j$,
we may assume that $\mu_i\neq \mu(H_j)$ for any $i,j$.
Thus either $\mu_0<\mu(H_1)$ or $\mu_0>\mu(H_1)$.

Assume that $\mu_0>\mu(H_1)$. By Lemma \ref{regular1}, we may assume that $\mathcal{A}(\mu_i)$
contains a regular module $M$ such that $|M|>2|\delta|$ for some $i$.
We may write $M=X_s$ for some quasi-simple module $X$ and $s>2R_X$.
On the other hand,  $\mu_j>\mu(H_1)$ for all $j$. Therefore,
$\mu(X_{j+1})$ is a direct successor of $\mu(X_j)$ for all $j\geq 2R_X$ (Proposition \ref{ds}).
It follows that $\mu(X_{s+j})=\mu_{r+j}$.
Note that  there does not exsit an indecomposable preinjective module
$M$ with GR measure $\mu(M)=\mu(X_{s+j})$ for any $j\geq 0$ (Corollary \ref{ds1}).

If $\mu_0<\mu(H_1)$, then $\mu_j<\mu(H_1)$ for all $j$. Since there are
only finitely many indecomposable preinjective module $N$ with $\mu(N)<\mu(H_1)$ (Proposition \ref{bigprop}),
we may obtain some $i$ such that $\mathcal{A}(\mu_j)$ consists of regular modules for each $j\geq i$.

The proof is completed.
\end{proof}

\subsection{$\mathbb{Z}$-indexed GR segments}

Let $\mathcal{S}_{\mathbb{Z}}$ be a $\mathbb{Z}$-indexed GR segment:
$$\ldots<\mu_{-3}<\mu_{-2}<\mu_{-1}<\mu_0<\mu_1<\mu_2<\mu_3<\ldots$$
We describes $\mathcal{A}(\mu_i)$ for $i$ smaller enough.
\begin{lemm}\label{preinj}
There is some $r$ such that $\mathcal{A}(\mu_i)$ contains only
preinjective modules for all $i<r$.
\end{lemm}
\begin{proof} Let $X$ be an indecomposable regular module such that
$\mu(X)\in\mathcal{S}_{\mathbb{Z}}$ and such that $|X|$ is minimal.
Without loss of generality we may assume that $\mu(X)=\mathcal{A}(\mu_0)$.
If there is an indecomposable regular module $Y$ with $\mu(Y)<\mu(X)=\mu_0$
and $\mu(Y)\in\mathcal{S}$, then  $|Y|<|X|$ by Lemma \ref{s=1}.
This is a contradiction. Thus $\mathcal{A}(\mu_{-i})$ contains no regular modules for any $i>0$.
\end{proof}

\begin{coro}For all $i$,
 $\mu_i>\mu(H_1)$.
\end{coro}

\begin{proof} Since there are only finitely many indecomposable preinjective
modules with GR measures smaller than $\mu(H_1)$, we have $\mu_{-i}>\mu(H_1)$
for $i>0$ larger enough by previous lemma.  Thus $\mu_i>\mu(H_1)$ for all $i$.
 \end{proof}

Let $a$ be the number of the isomorphism classes of the exceptional
quasi-simple modules $X$ whose GR measures satisfy $\mu(X_{R_X})\geq \mu(H_1)$.
\begin{prop} The number of the $\mathbb{Z}$-indexed GR segments is bounded by $a$.
\end{prop}
\begin{proof} Assume that $\mathcal{S}_{\mathbb{Z}}$ is a $\mathbb{Z}$-indexed GR segment.
Since $\mu_i>\mu(H_1)$, there is a $\mu_i\in\mathcal{S}_{\mathbb{Z}}$ such that
$\mathcal{A}(\mu_i)$ contains a regular module $X_s$ and $\mu_{i+j}=\mu(X_{s+j})$
for some quasi-simple module $X$.
Therefore, $\mathcal{S}_{\mathbb{Z}}$ gives
(not unique in general)  an exceptional quasi-simple $X$ such that
$\mu(X_{R_X})>\mu(H_1)$. It is clear that  different $\mathbb{Z}$-indexed GR segments
correspond to non-isomorphic quasi-simple modules.
Therefore, there are at most $a$ $\mathbb{Z}$-indexed GR segments.
\end{proof}

\subsection{$\mathbb{N}$-indexed GR segments}
Corollary \ref{homo} shows the existence of an $\mathbb{N}$-indexed GR segment.
It was already proved in \cite{Ch4} that for a tame quiver there are, but
only finitely many, $\mathbb{N}$-indexed GR segments. However, an upper bound of the number of
this kind of GR segments is still missing.  Similar to the
discussion for $\mathbb{Z}$-indexed GR segments, we will describe the
$\mathbb{N}$-indexed GR segments containing central measures and give an upper bound of the number.
Let $\mathcal{S}_{\mathbb{N}}$: $\mu_0<\mu_1<\mu_2<\ldots$ be an $\mathbb{N}$-indexed
GR segment ($\mu_0$ has no direct predecessor), which contains no take-off measures.

\begin{lemm}\label{preinj2} Only finitely many $\mathcal{A}(\mu_i)$
contains preinjective modules.
\end{lemm}
\begin{proof} This is just a restatement of Lemma \ref{regular2}.
\end{proof}

\begin{prop}
There is some $i>0$ and some quasi-simple module $X$ such that
$\mu_{i+j}=\mu(X_{t+j})$ for some $t\geq 1$ and all $j\geq 0$.
\end{prop}

\begin{proof} We may assume without loss of generality that
$\mathcal{A}(\mu_i)$ only contain regular modules for all $i$.
By Proposition \ref{ds}, we may also select $s>0$ such that, for any quasi-simple module $X$,
$\mu(X_{j+1})$ is a direct successor of $\mu(X_j)$ for any $j\geq s$.
Let $i>0$ such that $|\mu_j|>|X_s|$ for all $j\geq i$ and all exceptional quasi-simple modules $X$.
Assume that $X$ is a quasi-simple module such that $X_t\in\mathcal{A}(\mu_{i})$,
then $X_{t+j}\in\mathcal{A}(\mu_{i+j})$ for all $j
\geq 0$. In particular $\mathcal{S}_{\mathbb{N}}$ gives rise  to a quasi-simple module $X$.
\end{proof}

{\it Proof of Theorem}. A $\mathbb{Z}$-indexed GR segment or an
$\mathbb{N}$-indexed GR segment that contains central measures, which are not of the forms $\mu(H_i)$,
gives rise to (may not be unique) an exceptional quasi-simple module. Moreover, different such GR segments
correspond to non-isomorphic quasi-simple modules. Thus the number of these kinds of GR segments is bounded by
$b$, the number of the isomorphism classes of the exceptional quasi-simple modules.
On the other hand, all $\mu(H_i)$ are contained in the same $\mathbb{N}$-indexed GR segment.
Thus the central part of
a tame quiver contains at most $b+1$ GR-segments. Note that the take-off part is also $\mathbb{N}$-indexed
and the landing part is $-\mathbb{N}$-indexed.
Therefore, a tame quiver has at most $b+3$ GR-segments.

\subsection{Examples}

(1) Let $Q$ be a tame quiver of type $\widetilde{\mathbb{A}}_n$ with sink-source orientation.
If $n=1$, i.e., $Q$ is a Kronecker quiver, then there is precisely one $\mathbb{N}$-indexed GR segment,
which consists of the GR measures
$\mu(H_i)$ of homogeneous modules. If $n>1$, then the central part contains only  two $\mathbb{N}$-indexed GR segments:
one is, as above, consisting of GR measures of homogeneous $\mu(H_i)$, and the other one is of the form $\{\mu(X_i)\}$, where $X$ is any exceptional quasi-simple module.
Note that in this case, there are no $\mathbb{Z}$-indexed GR segments.

\bigskip
(2) Let $Q$ be the following quiver:
$$
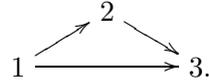
\xymatrix@R=6pt@C=20pt{&2\ar[rd]&\\1\ar[ru]\ar[rr]&& 3.\\
}$$
Let $X$ be, up to isomorphism, the unique quasi-simple of length $2$.
We denote by $M^i$ the unique (up to isomorphism)
indecomposable preinjective module with length $3i+2$.
The only $\mathbb{Z}$-indexed GR segment is the following:
$$\ldots<\mu(M^i)<\ldots<\mu(M^2)<\mu(M^1)=\mu(X_3)<\mu(X_4)<\mu(X_5)<\ldots<\mu(X_j)<\ldots$$
In the central part, there is precisely one $\mathbb{N}$-indexed GR segment which is given by the GR measures
of homogeneous modules:
$$\mu(X_2)=\mu(H_1)<\mu(H_2)<\ldots<\mu(H_i)<\ldots$$
We refer to \cite{R2} for  details of the description of the GR measures of this quiver.

\bigskip
We may characterize the tame quivers of type $\widetilde{\mathbb{A}}_n$,
which admit $\mathbb{Z}$-indexed GR segments.
\begin{prop}Let $Q$ be a tame quiver of type $\widetilde{\mathbb{A}}_n$. Then the following are equivalent:
\begin{itemize}
 \item[(1)] $Q$ is not equipped with a sink-source orientation.
 \item[(2)] There are preinjective central modules.
 \item[(3)] There are infinitely many isomorphism classes of  preinjective central modules.
 \item[(4)] There exists a $\mathbb{Z}$-indexed GR segment.
\end{itemize}
\end{prop}
\begin{proof}
The equivalences of the first three statements were already shown in \cite{Ch4}.

(4) implies (3) is obvious by Lemma \ref{preinj}.
Conversely, assume that statement (3) holds.
Let $\mathcal{A}=\cup_{\mu}\mathcal{A}(\mu)$, where the union is taken over all central GR measures $\mu\in\mathcal{S}_{\mathbb{N}}$ for some $\mathbb{N}$-indexed GR segments $\mathcal{S}_{\mathbb{N}}$.
This is a finite union
since the main theorem gives an upper
bound of the number of the $\mathbb{N}$-indexed GR segments.  On the other hand,
Lemma \ref{preinj2} implies that in each $\mathbb{N}$-indexed GR segment, there are only finitely
GR measures $\mu$ such that $\mathcal{A}(\mu)$ contains (finitely many) preinjective modules.
It follows that $\mathcal{A}$ contains only finitely many preinjective modules.
Therefore, there exists a $\mathbb{Z}$-indexed GR segment by (3) and the fact that a GR segment in the central part
is either $\mathbb{N}$- or $\mathbb{Z}$-indexed.
\end{proof}

\end{document}